\documentclass[11pt]{amsart}
\usepackage{amsfonts,amsthm,amssymb, mathrsfs}

\DeclareMathOperator\diag{diag}
\DeclareMathOperator\codim{codim}
\DeclareMathOperator\Span{Span}

\DeclareMathOperator\HH{HH}
\DeclareMathOperator\ext{Ext}
\DeclareMathOperator\cac{ClassAlgConst}
\DeclareMathOperator\supp{supp}

\newcommand{\CC}{\mathbb{C}}
\newcommand{\ZZ}{\mathbb{Z}}
\newcommand{\ds}{\displaystyle}
\newcommand{\vol}[1]{\text{vol}_{#1}^{\perp}}
\newcommand{\mol}{\bullet}
\newcommand{\HHSSG}[1]{\HH^{#1}(S(V),S(V)\#G)}
\newcommand{\conjG}{\mathscr{C}_G}

\theoremstyle{plain}
\newtheorem{theorem}[subsection]{Theorem}
\newtheorem*{theorem*}{Theorem}
\newtheorem{proposition}[subsection]{Proposition}
\newtheorem{lemma}[subsection]{Lemma}

\newtheorem{observation}[subsection]{Observation}
\newtheorem{corollary}[subsection]{Corollary}

\theoremstyle{definition}
\newtheorem{definition}[subsection]{Definition}
\newtheorem*{remarks}{Remarks}
\newtheorem*{remark}{Remark}
\newtheorem{example}[subsection]{Example}

\title[Comparing Codimension and Absolute Length]
{Comparing Codimension and Absolute Length in Complex Reflection Groups}
\author{Briana Foster-Greenwood}
\date{\today}

\address{Department of Mathematics, University of North Texas,
Denton, Texas 76203, USA}
\email{BrianaFoster-Greenwood@my.unt.edu}

\thanks{2010 {\it Math Subject Classification: }05E15, 20F55, 06A06, 16E40.}

\begin{document}

\begin{abstract}
Reflection length and codimension of fixed point spaces induce
partial orders on a complex reflection group.
While these partial orders are of independent combinatorial interest,
our investigation is motivated by a connection between the codimension order
and the algebraic structure of cohomology governing deformations of
skew group algebras.
In this article, we compare the reflection length and codimension
functions and discuss implications for cohomology of skew group
algebras. 
We give algorithms using character theory for computing reflection length, atoms, 
and poset relations.
Using a mixture of theory, explicit examples, and computer calculations
in GAP, we show that Coxeter groups and the infinite family $G(m,1,n)$
are the only irreducible complex reflection groups 
for which the reflection length and codimension orders coincide.
We describe the
atoms in the codimension order for the infinite family $G(m,p,n)$,
which immediately yields an explicit description of generators for 
cohomology.
\end{abstract}

\maketitle

\section{Introduction}\label{INTRO}

The theory of real and complex reflection groups blends 
techniques and ideas from combinatorics,
representation theory, discrete geometry, invariant theory, and hyperplane
arrangements.  For example, various Hecke algebras are built from the
geometry of a Weyl group action.  Current research focuses on
Hecke algebras and variants constructed from complex reflection groups.
These deformations of algebras use both group theory and geometry.
In this article, we explore two partial orders on reflection groups:
a reflection length order (related to the word metric of geometric group theory)
and the codimension order (capturing the geometry of
the group action).
Various reflection length orders are key tools in the theory
of Coxeter groups, while codimension appears in the
numerology of complex reflection groups and in
recent work connecting combinatorics and Hochschild cohomology.

While the combinatorics of the reflection length and codimension
posets is of interest in its own right, our current motivation lies
in a connection to Hochschild cohomology of certain skew group algebras.
Given a finite group $G$ acting on a vector space $V$, one may form
a skew group algebra by taking the semi-direct product 
of the group algebra $\mathbb{C}G$ with the algebra of polynomial functions on $V^{*}$.
Formal deformations of skew group algebras
include graded Hecke algebras
(also referred to in the literature as
symplectic reflection algebras,
rational Cherednik algebras, and
Drinfeld Hecke algebras), which have been studied by
Lusztig~\cite{Lusztig}, 
Drinfeld~\cite{Drinfeld}, 
Etingof and Ginzburg~\cite{EtingofGinzburg}, 
and others.
Gordon~\cite{Gordon} used graded Hecke algebras to prove an analogue of 
the $n!$-conjecture for Weyl groups suggested by Haiman.

Hochschild cohomology detects potential associative deformations
of an algebra.  
The Hochschild cohomology ring of a skew group algebra may be identified with the
$G$-invariant subalgebra of a larger cohomology ring, which 
Shepler and Witherspoon~\cite{SheplerWitherspoonCup} show is finitely
generated (under cup-product) by elements corresponding to
atoms in the codimension order on $G$.

Reflections are always atoms in the codimension poset, as they have codimension one.  
Are these the only atoms in the codimension poset?  
In case $G$ acts by a reflection representation, we can answer this
question by comparing the reflection length and codimension functions.
The question has been answered in the affirmative 
for Coxeter groups (see Carter~\cite{Carter}) 
and the infinite family $G(m,1,n)$ (see Shi~\cite{ShiJ} or Shepler and Witherspoon~\cite{SheplerWitherspoonCup}).
In this article, we use a mixture of theory,
explicit examples, and computer calculations using the software GAP to show
reflection length and codimension {\it do not} coincide 
in the remaining monomial reflection groups $G(m,p,n)$ (see Section~\ref{GMPN})
and exceptional complex reflection groups (see Sections~\ref{RANKTWO} and \ref{EXCEPTIONAL}):
\medskip
\begin{theorem*}
   Let $G$ be an irreducible complex reflection group.  The reflection
   length and codimension functions coincide if and only if $G$ is a Coxeter
   group or $G=G(m,1,n)$.
\end{theorem*}
\medskip
When $G$ is not a Coxeter group or a monomial reflection group
$G(m,1,n)$, cohomology prompts us to look further to find the 
nonreflection atoms in the codimension
poset.  For the groups $G(m,p,n)$, we give an explicit combinatorial
description of the atoms in the codimension poset.
In the rank two exceptional complex reflection groups, 
we characterize the nonreflection atoms as the elements with reflection length
exceeding codimension.  For the remaining groups, we provide a 
count of the nonreflection atoms in the codimension poset.

In Section~\ref{COHOMOLOGY} we apply our results 
on the atoms in the codimension poset to obtain
information about the degrees and support of generators for Hochschild cohomology rings
arising in deformation theory. 

\subsection*{Acknowledgements} The author thanks Ph.D. advisor Anne Shepler
for suggesting this project and for helpful discussions.
The author also thanks Cathy Kriloff for suggestions on an early draft.
\\
\section{Reflection length and codimension posets}\label{POSETS}

\subsection*{Definitions}
Let $V$ be an $n$-dimensional vector space over $\mathbb{R}$ or $\mathbb{C}$.
An element $g$ in $GL(V)$ is a {\sffamily\bfseries reflection} if it has finite
order and fixes a hyperplane pointwise.
A {\sffamily\bfseries reflection group} is a finite group $G \subset GL(V)$ generated
by reflections.  
We define two class functions on a reflection group $G$  
and use these functions to partially order the group.
The {\sffamily\bfseries(absolute) reflection length}
of an element $g$ is the minimum number of
factors needed to 
write $g$ as a product of reflections:
$$
\ell(g)=\min\{k\::\:g=s_1\cdots s_k \textrm{ for some reflections $s_1,\ldots,s_k$ in $G$}\}.
$$
We set $\ell(1)=0$.  Note that this function gives length
with respect to \emph{all} reflections in the group, as opposed to
a set of fundamental or simple reflections. 
Each reflection in $G$ fixes a hyperplane pointwise,
and each remaining element fixes an intersection of hyperplanes. 
The {\sffamily\bfseries codimension} function on $G$ keeps track of the 
codimension of each fixed point space:
$$\codim(g)=n-\dim \{v \in V\::\:gv=v\}.$$
The reflection length and codimension functions satisfy the following properties:
\begin{itemize}
  \item constant on conjugacy classes
  \item subadditive: $\ell(ab)\leq\ell(a)+\ell(b)$ and $\codim(ab)\leq\codim(a)+\codim(b)$
  \item $\codim(g)\leq\ell(g)$ for all $g$ in $G$.
\end{itemize}

Now define the reflection length order on $G$ by
$$a \leq_{_{\ell}} c
\hspace{.25in}\Leftrightarrow\hspace{.25in}
\ell(a)+\ell(a^{-1}c)=\ell(c).
$$
Analogously\footnote{Brady and Watt~\cite{BradyWatt} prove $\leq_{_\perp}$
is a partial order.  Their proof is also valid when codimension is
replaced by any function $\mu:G\rightarrow[0,\infty)$ 
satisfying $\mu(a)=0$ iff $a=1$ (positive definite) and
$\mu(ab) \leq \mu(a)+\mu(b)$ for all $a$, $b$ in $G$ (subadditive).},
define the codimension order on $G$ by 
$$a \leq_{_\perp} c
\hspace{.25in}\Leftrightarrow\hspace{.25in}
\codim(a)+\codim(a^{-1}c)=\codim(c).
$$
Since reflection length and codimension are constant on conjugacy classes,
we get induced partial orders on the set of conjugacy classes of $G$.
Define the reflection length (likewise codimension) of a 
conjugacy class to be the reflection length (likewise codimension) of the elements
in the conjugacy class.  If $A$ and $C$ are conjugacy classes
of $G$, then set $A \leq_{_{\ell}} C$ if there exists an element $a \in A$
and an element $c \in C$ with $a \leq_{_{\ell}} c$.  Analogously, 
define the codimension order on conjugacy classes.
In Section~\ref{EXCEPTIONAL}, we will appeal to character theory to
deduce information about the partial orders on $G$ by working
with the (somewhat simpler) partial orders on the set of
conjugacy classes of $G$.

In a poset $(P,\leq)$, we say $b$ {\sffamily\bfseries covers} $a$
if $b > a$ and the interval \mbox{$\{x\in P\::\:a<x<b\}$} is empty.
The {\sffamily\bfseries atoms} of a poset are the covers of the minimum
element (when it exists).
The identity is the minimum element in the reflection length poset
and in the codimension poset.
For emphasis, we often refer to the atoms in
the codimension poset as {\sffamily\bfseries codimension atoms}.
Note that an element $a$ in $G$ is an atom in the poset on $G$
if and only if its conjugacy class is an atom in the corresponding
poset on the set of conjugacy classes of $G$. 

\subsection*{Functions versus posets}
We now show that 
comparing the length and codimension functions
is (in a sense) equivalent to comparing the set of atoms in each poset.

\begin{definition}
  We say $g=g_1 \cdots g_k$ is a {\sffamily\bfseries factorization of $g$ with codimensions adding}
  if $\codim(g)=\codim(g_1)+\cdots+\codim(g_k)$. 
\end{definition}

\noindent Note that if $g=g_1 \cdots g_k$ is a factorization with
codimensions adding, then, using the fact that codimension is subadditive
and constant on conjugacy classes, we also have $g_1,\ldots,g_k \leq_{_\perp} g$.
Furthermore, since $V$ is finite dimensional, we can work recursively
to factor any nonidentity element of $G$ into a product of codimension atoms with
codimensions adding:

\begin{observation}
  Given a nonidentity group element $g$, there exist codimension atoms \mbox{$a_1,\ldots,a_k \leq_{_\perp} g$}
  such that \mbox{$g=a_1\cdots a_k$} and $\codim(g)=\codim(a_1)+\cdots+\codim(a_k)$.
\end{observation}

The next two lemmas follow from repeated use of subadditivity of length and codimension
and the fact that codimension is bounded above by reflection length.

\begin{lemma}\label{forward}
   Fix $g$ in $G$.  If $\ell(a)=\codim(a)$ for every codimension atom $a \leq_{_\perp} g$, then $\ell(g)=\codim(g)$.
\end{lemma}   
\begin{proof}
  The statement certainly holds for the identity.
  Now let $g=a_1\cdots a_k$ be a factorization of $g \neq 1$ into atoms with codimensions adding.
  (Note that necessarily $a_{1},\ldots,a_{k} \leq_{_\perp} g$.)
  Then
    $$
     \ell(g)  \leq  \ell(a_1)+\cdots+\ell(a_k) 
              =     \codim(a_1)+\cdots+\codim(a_k) 
              =     \codim(g)
              \leq  \ell(g),$$
  with equality throughout.
\end{proof}  

\begin{lemma}\label{backward}   
   Let $g \in G$ with $\ell(g)=\codim(g)$.  If $h \leq_{_\ell} g$, then $h \leq_{_\perp} g$.
\end{lemma}   
\begin{proof}
  Subadditivity gives $\codim(g) \leq \codim(h)+\codim(h^{-1}g)$.  
  If $\ell(g)=\codim(g)$ and $h \leq_{_\ell} g$, we also have the reverse inequality:
  $$
  \codim(h)+\codim(h^{-1}g) \leq \ell(h) + \ell(h^{-1}g) = \ell(g) = \codim(g).
  $$
\end{proof}

Lemma~\ref{forward} and Lemma~\ref{backward} combine to
reveal that the reflection length and codimension functions 
coincide on all of $G$ if and only if every codimension atom
is a reflection.

\begin{proposition}\label{tfae}
  The following are equivalent:
  \begin{enumerate}
     \item $\ell(g)=\codim(g)$ for every $g$ in $G$.
     \item $\ell(g)=\codim(g)$ for every codimension atom $g$ in $G$.
     \item Every codimension atom is a reflection.
     \item For every $g\neq1$, there exists a reflection $s$ in $G$ such
           that $\codim(gs)<\codim(g)$.
  \end{enumerate}
\end{proposition}

\begin{proof}
   The implication (1) $\Rightarrow$ (2) is immediate.  
   Application of Lemma~\ref{backward} with $g$ a codimension atom and $h$ a reflection
   shows (2) $\Rightarrow$ (3).  Lastly, if every codimension atom is a 
   reflection, then the hypothesis of Lemma~\ref{forward} holds for each $g$
   in $G$, and hence (3) $\Rightarrow$ (1).   
   It is straightforward to work (4) into the loop via (1) $\Rightarrow$ (4) and (4) $\Rightarrow$ (3).
\end{proof}

\section{The infinite family $G(m,p,n)$}\label{GMPN}

The group $G(m,1,n)\cong(\ZZ/m\ZZ)^n\rtimes \mathfrak{S}_n$ consists of all $n \times n$ monomial matrices having $m^{th}$ roots
of unity for the nonzero entries.
For $p$ dividing $m$, the group $G(m,p,n)$ is the subgroup of $G(m,1,n)$ consisting
of those elements whose nonzero entries multiply to an $(\frac{m}{p})^{th}$
root of unity.
Throughout this section let $\zeta_m=e^{2\pi i/m}$.
Each group $G(m,p,n)$ contains the order two {\it transposition type reflections}
of the form $\delta\sigma$,
where $\sigma$ is a transposition 
swapping the $i^{th}$ and $j^{th}$ basis vectors, and 
$\delta=\diag(1,\ldots,\zeta_m^a,\ldots,\zeta_m^{-a},\ldots,1)$
scales rows $i$ and $j$ of $\sigma$ ($\delta=1$ is a possibility).  
When $p$ properly divides $m$, the group $G(m,p,n)$ also contains the
{\it diagonal reflections} $\diag(1,\ldots,\zeta_m^{a},\ldots,1)$
where $0<a<m$ and $p$ divides $a$.
The $G(m,p,n)$ family includes the following Coxeter groups ($n\geq2$):
\begin{itemize}
   \item symmetric group: $G(1,1,n)$ (not irreducible)
   \item Weyl groups of type $B_n$ and $C_n$: $G(2,1,n)$
   \item Weyl groups of type $D_n$: $G(2,2,n)$
   \item dihedral groups: $G(m,m,2)$
\end{itemize}

In this section, we describe the atoms in the codimension poset
for an arbitrary group $G(m,p,n)$.  In the groups for which the
reflection length and codimension functions do not coincide,
we give explicit examples of elements with length exceeding codimension.

\begin{definition}
  Let $V=V_1\oplus\cdots\oplus V_n$ be a decomposition of $V\cong\CC^n$
  into one-dimensional subspaces permuted by $G(m,p,n)$.  Let $g$ be in $G(m,p,n)$,
  and partition $\{V_1,\ldots,V_n\}$ into $g$-orbits, say 
  $\mathcal{O}_1,\ldots,\mathcal{O}_r$.  The action of $g$ on $\bigoplus_{V_{j}\in\mathcal{O}_i}V_j$
  can be expressed as $\delta_i\sigma_i$, where $\delta_i$ is diagonal and $\sigma_i$
  is a cyclic permutation. 
  (Thus, up to conjugation by a permutation matrix,
  $g$ is block diagonal with $i^{th}$ block $\delta_i\sigma_i$.)  
  The {\sffamily\bfseries cycle-sum} of $g$ corresponding to orbit $\mathcal{O}_i$ is
  the exponent $c_i$ (well-defined modulo $m$) such that \mbox{$\det(\delta_i)=\zeta_m^{c_i}$}.    
\end{definition}

Cycle-sums allow us to quickly read off codimension of an element:
$$\codim(g)=n-\#\{i\::\:c_i\equiv0\pmod m\}.$$

Note that for a reflection $t$ and any group element $g$, 
the relation $t\leq_{_{\perp}}g$ is
equivalent to $\codim(t^{-1}g)=\codim(g)-1$.
Letting $s=t^{-1}$ and noting that the conjugate elements
$sg$ and $gs$ have the same codimension,
we obtain the following convenient observation:

\begin{observation}
  An element $g\neq1$ is comparable with a reflection in the codimension poset
  if and only if
  there exists a reflection $s$ such that $\codim(gs)<\codim(g)$.
\end{observation}

We recall from Shi~\cite{ShiJ} (see Corollary 1.8 and the proof of Theorem 2.1)
the three possibilities for how the cycle-sums
change upon multiplying by a reflection:

\begin{lemma}[Shi~\cite{ShiJ}]\label{cyclesums}
   Let $g \in G(m,p,n)$ with cycle-sums $c_1,\ldots,c_r$ corresponding to
   $g$-orbits $\mathcal{O}_1,\ldots,\mathcal{O}_r$.
   If $s$ is a transposition type reflection interchanging $V_i$ and $V_j$, then
   the cycle-sums of $g$ split or merge into the cycle-sums of $gs$:
   \begin{enumerate}
      \item If $V_{i}$ and $V_{j}$ are in the same $g$-orbit, say $\mathcal{O}_k$,
            then $gs$ has cycle sums\\
            \mbox{$c_1,\ldots,\widehat{c_{k}},\ldots,c_r,d,c_k-d$} for some integer $d$. 
            \smallskip
      \item If $V_{i}$ and $V_j$ are in different $g$-orbits,
            say $\mathcal{O}_k$ and $\mathcal{O}_l$, 
            then $gs$ has cycle sums
            $c_1,\ldots,\widehat{c_{k}},\ldots,\widehat{c_l},\ldots,c_r,c_k+c_l$.
            \smallskip
   \end{enumerate}
   Let $s$ be a diagonal reflection scaling $V_i$ by non-1 eigenvalue $\zeta_m^{a}$
            (where $p$ divides $a$).
   \begin{enumerate}
      \addtocounter{enumi}{2}
      \item If $V_i$ is in the $g$-orbit $\mathcal{O}_k$, 
            then $gs$ has cycle sums $c_1,\ldots,c_k+a,\ldots,c_r$.
   \end{enumerate}
\end{lemma}

Note that if $g$ is nondiagonal, then
by choosing a suitable transposition type reflection $s$, 
we can arrange for the cycle-sum $d$ of $gs$ 
in part (1) of Lemma~\ref{cyclesums} to be any of \mbox{$0,\ldots,m-1$}.
In particular, we can choose $s$ so that $d=0$, 
thereby increasing the number of zero cycle-sums and
decreasing codimension.
Hence {\it every nonreflection atom in the codimension poset must be diagonal}. 
The converse is false,
but we come closer to the set of nonreflection atoms 
by considering only $p$-connected diagonal elements.

\begin{definition}
  A diagonal matrix $g\neq1$ whose non-1 eigenvalues are $\zeta_{m}^{c_1},\ldots,\zeta_{m}^{c_k}$
  (listed with multiplicities) is
  {\sffamily\bfseries $p$-connected} if $p$ divides $c_1+\cdots+c_k$ but 
  $p$ does not divide $\sum_{i \in I}c_i$ for $I \subsetneq \{1,\ldots,k\}$.
  (Note that $g$ is in $G(m,p,n)$ iff $p$ divides $c_1+\cdots+c_k$.)
\end{definition}

It is easy to see that each nonidentity diagonal element of $G(m,p,n)$
factors in $G(m,p,n)$ into $p$-connected elements with codimensions adding.
Thus {\it every nonreflection atom in the codimension poset must be $p$-connected}.
We next check for poset relations among the reflections and $p$-connected elements.

\begin{lemma}
  The $p$-connected elements of $G(m,p,n)$
  are pairwise incomparable in the codimension poset.
\end{lemma}

\begin{proof}
  Suppose $a,b$ in $G(m,p,n)$ are diagonal elements such that $ab$ is
  $p$-connected and $\codim(a)+\codim(b)=\codim(ab)$.
  Since codimensions add, it is not hard to show that the non-1 eigenvalues
  of $ab$ are the non-1 eigenvalues $\zeta_{m}^{a_1},\ldots,\zeta_m^{a_{\codim(a)}}$ of $a$
  together with the non-1 eigenvalues $\zeta_m^{b_1},\ldots,\zeta_m^{b_{\codim(b)}}$ of $b$.
  If $a,b\neq1$, we have a contradiction to
  $p$-connectedness of $ab$, as $p$ divides $a_1+\cdots+a_{\codim(a)}$
  by virtue of $a$ being in $G(m,p,n)$.
\end{proof}


\begin{lemma}
  Let $g$ in $G(m,p,n)$ be $p$-connected and not a diagonal
  reflection.  Then there exists a
  reflection $s\leq_{_\perp}g$ if and only if $g$ has codimension two and
  non-1 eigenvalues $\zeta_m^c$ and $\zeta_m^{-c}$ for some $c$. 
\end{lemma}
\begin{proof}
  If $g$ has codimension two and non-1 eigenvalues $\zeta_m^{c}$ and $\zeta_m^{-c}$,
  then $g$ factors into two reflections with codimensions adding.
  The factorization in dimension two illustrates the general case:
  $$
  \left(
   \begin{array}{cc}
      \zeta_m^c &             \\
        & \zeta_m^{-c}           \\
   \end{array}
   \right)
   =
   \left(
   \begin{array}{cc}
       &  \zeta_m^{c}           \\
      \zeta_m^{-c}  &            \\
   \end{array}
   \right)
   \left(
   \begin{array}{cc}
        & 1            \\
      1 &            \\
   \end{array}
   \right).
  $$
  Conversely, if $g$ is comparable with a reflection, then there must be a 
  reflection $s$ with $\codim(gs)<\codim(g)$.
  If $s$ is a diagonal type reflection, then we get a contradiction
  to $p$-connectedness of $g$.  If $s$ is a transposition type
  reflection, then, since $g$ is diagonal, we use Lemma~\ref{cyclesums} (2)
  to see that $g$ must have nonzero cycle-sums $c_k$ and $c_l$ such that
  $c_k+c_l\equiv0 \pmod m$.  By $p$-connectedness, $c_k$ and $c_l$
  must be the only nonzero cycle-sums of $g$, and hence the only non-1
  eigenvalues of $g$ are $\zeta_m^{c_k}$ and $\zeta_m^{c_l}=\zeta_m^{-c_k}$.
\end{proof}

Since every element of $G(m,p,n)$ must be above {\it some} atom,
we now have the collection of codimension atoms:

\begin{proposition}\label{GmpnAtoms}
  The codimension atoms for $G(m,p,n)$ are the reflections together
  with the $p$-connected elements except for those with codimension two
  and determinant one.
\end{proposition}

It is known that length and codimension coincide for Coxeter groups
and the family $G(m,1,n)$, 
which, incidentally, includes the rank one groups $G(m,p,1)=G(\frac{m}{p},1,1)$.
For the remaining groups in the family $G(m,p,n)$, we give
explicit examples of codimension atoms with reflection length exceeding
codimension.

\begin{corollary}\label{gmpnnope}
  The reflection length and codimension functions do not coincide in the
  following groups:
  \begin{itemize}
     \item $G(m,p,n)$ with $1<p<m$ and $n\geq2$
     \item $G(m,m,n)$ with $m \geq 3$ and $n\geq3$.
  \end{itemize}
\end{corollary}
\begin{proof}
Let $I_k$ be the $k \times k$ identity matrix, and let $M_2$ and $M_3$ be the matrices
$$
   M_2=\left(
   \begin{array}{cc}
      \zeta_{_{m}} &             \\
       & \zeta_{_{m}}^{^{p-1}}   \\
   \end{array}
   \right)
   \textrm{ and }
   M_3=\left(
   \begin{array}{ccc}
      \zeta_{_{m}} &                     &              \\
                   & \zeta_{_{m}}^{^{-2}} &             \\
                   &                     & \zeta_{_{m}} \\
   \end{array}
   \right).
$$
In $G(m,p,n)$ with $1<p<m$ and $n \geq 2$,
the direct sum matrix $M_2 \oplus I_{n-2}$
has reflection length three and codimension two.
In $G(m,m,n)$ with $m \geq 3$ and $n \geq 3$, 
the direct sum matrix $M_3 \oplus I_{n-3}$
has reflection length four and codimension three.
\end{proof}

\begin{remarks}\hfill
\begin{itemize}
  \item A $1$-connected element must be a diagonal reflection,
        so the set of codimension atoms in $G(m,1,n)$ is simply
        the set of reflections.  By Lemma~\ref{tfae}, this recovers the result that
        length and codimension coincide in $G(m,1,n)$.
  \item Shi \cite{ShiJ} gives a formula for reflection length in $G(m,p,n)$
        in terms of a maximum over certain partitions of cycle-sums.
        He also uses existence of a certain partition of the cycle-sums
        as a necessary and sufficient 
        condition for an element to have reflection length equal to codimension.
\end{itemize}
\end{remarks}

\section{Rank two exceptional reflection groups}\label{RANKTWO}

The complex reflection groups $G_{4}-G_{22}$ act irreducibly on $V\cong\mathbb{C}^{2}$.
Each has at least one conjugacy class of elements 
for which length and codimension differ.
In all except for $G_{8}$ and $G_{12}$, 
an argument comparing the order of the reflections with the order
of the center of the group demonstrates the existence of a 
central element with length greater than codimension.

\begin{lemma}\label{zorder}
  Let $G$ be an irreducible complex reflection group acting on $V\cong\CC^2$.
  If $z\in G$ is central and $G$ does not contain any
  reflections with the same order as $z$, then $\ell(z) > \codim(z)$.
\end{lemma}

\begin{proof}
  Since $G$ acts irreducibly on $V \cong \CC^2$, 
  each central element $z \neq 1$ is represented by a scalar matrix 
  of codimension two.
  Note that if $z=st$ is a product of two reflections, then
  $s$ and $t$ are actually {\it commuting} reflections.
  Then, working with $s$, $t$, and $z$ simultaneously in diagonal form,
  it is easy to deduce that the reflections $s$ and $t$ must have
  the same order as $z$.  Thus if $G$ does not contain any reflections of the same order
  as $z$, we have $\ell(z) > 2$.
\end{proof}

Note that if $g$ is a group element with $\codim(g)=2$, then
$\ell(g)=\codim(g)$ if and only if $g$ can
be expressed as a product of two reflections.  Thus, we describe
the codimension atoms for a rank two reflection group:

\begin{lemma}
  The codimension atoms in a rank two complex reflection
  group are the reflections together with the elements $g$ 
  such that $\ell(g)>\codim(g)$.
\end{lemma}

\begin{proposition}\label{ranktwonope}
   Reflection length and codimension do not coincide in the
   rank two exceptional complex reflection groups $G_{4}-G_{22}$.
\end{proposition}

\begin{proof}
Inspection of Tables I, II, and III in Shephard-Todd~\cite{ShephardTodd} 
and application of Lemma~\ref{zorder} shows that each rank two group $G_{i}$ with $i \neq 8$ or $12$
has a central element $z$ with $\ell(z)>\codim(z)$.

The group $G_{8}$ can be generated by the order four reflections
$$
   r_1=\left(
   \begin{array}{cc}
      i &             \\
        & 1           \\
   \end{array}
   \right)
   \textrm{ and }
   r_2=\frac{1}{2}\left(
   \begin{array}{ccc}
     \phantom{-}1+i & 1+i       \\
     -1-i & 1+i       \\
   \end{array}
   \right).
$$
The element 
$$g=r_1(r_{1}r_{2}^2r_{1}^{-1})r_{2}
   =\frac{1}{2}\left(
   \begin{array}{ccc}
     -1+i & \phantom{-}1-i       \\
     -1-i & -1-i       \\
   \end{array}
   \right)
$$
has
length three and codimension two.  
Note that if $g$ were the product of two reflections, then
$gs$ would be a reflection for some reflection $s$ in $G_{8}$.
However, computation shows $\codim(gs)=2$ 
for all reflections $s$ in $G_{8}$.

  For $G_{12}$, let $S$ and $T$ be the generators given in 
  Shephard-Todd~\cite{ShephardTodd}.  Although $S$ is a reflection,
  the element $T$ has codimension two.
  We express $T$ as the product of two
  reflections (each a conjugate of $S$):
  $$T=(STST^{-1}S^{-1})(T^{-1}ST^{-1}STS^{-1}T).$$
  The element $ST$ has length three and codimension two.  (We verify the
  length by noting that all reflections in $G_{12}$ have determinant $-1$,
  so $ST$ also has determinant $-1$ and must have odd length.)

\end{proof}

\begin{remarks}  Carter~\cite{Carter} proves length equals codimension in Weyl groups.
   Although Carter's proof applies equally well to any Coxeter group, we
   indicate two places where the proof can break down for a general complex reflection group.
  \begin{itemize}
    \item Carter's proof shows that in a real reflection group, 
          if $g$ has maximum codimension, i.e., $\codim(g)=n$, then
          $\codim(gs) < n$ for all reflections $s$ in the group.
          This may fail in a general complex reflection group,
          as illustrated by the element $g$ in $G_8$ given above
          in the proof of Proposition~\ref{ranktwonope}.
    \item Though $G_{12}$ only has order two reflections, 
          Carter's proof fails for $G_{12}$ because a complex
          inner product is not symmetric. 
 \end{itemize}
\end{remarks}

\section{Exceptional reflection groups $G_{23}-G_{37}$}\label{EXCEPTIONAL}

For the exceptional reflection groups, we work with the partial
orders on the set $\conjG$ of conjugacy classes of $G$.  
With the aid of the software GAP~\cite{GAP}\nocite{CHEVIE}, we compute 
reflection length, atoms, and poset relations,
appealing to character theory to speed up the computations.
Some of the exceptional reflection groups are Coxeter groups, for which reflection length and codimension
are known to agree.  For the remaining groups, our computations show reflection length
and codimension do not coincide.
 
We first recall class algebra constants, 
which we use to aid our computations.
Let $X$, $Y$, and $C$ be conjugacy classes of $G$, and let
$c$ be a fixed representative of $C$.
The class algebra constant $\cac(X,Y,C)$
counts the number of pairs $(x,y)$ in $X \times Y$ such that $xy=c$.
These are the structure constants for the center of the group
algebra and have a formula in terms of the irreducible characters of $G$
(details can be found in James and Liebeck~\cite{JamesLiebeck}, for example).

Using class algebra constants, we can inductively find the elements
of each reflection length without having to multiply individual group elements.
Let $L(k)$ denote the set of conjugacy classes whose elements
have reflection length $k$.
Suppose conjugacy class $C$ is not in $L(0) \cup\cdots\cup L(k)$ so
that $\ell(C)$ is at least $k+1$.
Then $C$ is in $L(k+1)$ if and only if $\cac(X,Y,C)$ is
nonzero for some $X$ in $L(1)$ and $Y$ in $L(k)$.
Since the class algebra constants are nonnegative, we have
$\ell(C)=k+1$ if and only if
$\sum\{\cac(X,Y,C)\::\:X\in L(1)\text{ and }Y\in L(k)\} \neq 0.$

Using the same idea, we can easily compute all relations in the
reflection length and codimension posets on the set of conjugacy classes of $G$.
For example, in the codimension poset, we have 
$$A \leq_{_\perp} C  \phantom{ssss}\Leftrightarrow\phantom{ssss} 
\sum_{\stackrel{X \in \conjG}{\codim(A)+\codim(X)=\codim(C)}}\hspace{-.72in}\cac(A,X,C)\neq0.$$
In particular,
\begin{center}
{\renewcommand{\arraystretch}{1.5}
\begin{tabular}{ccp{3in}} 
$C$ is an atom in $(\conjG,\leq_{_\perp})$ & \phantom{ssss}$\Leftrightarrow$\phantom{ssss} &
$\ds\sum_{\stackrel{X,Y \in \conjG\backslash\{\{1\},C\}}{\codim(X)+\codim(Y)=\codim(C)}}\hspace{-.72in}\cac(X,Y,C)=0.$
\end{tabular}}
\end{center}

%

Table~\ref{tab:AtomCount} summarizes the data collected for the groups $G_{23}-G_{37}$.
(The Coxeter groups are included for contrast.)
The middle columns compare the number of conjugacy classes of nonreflection atoms
with the number of conjugacy classes $C$ such that $\ell(C)\neq\codim(C)$.
The final columns compare maximum reflection length with
the dimension $n$ of the vector space on which the group acts.  Note that in each
case the maximum reflection length is at most $2n-1$, usually less.

\begin{table}
\centering
   \tiny{\renewcommand{\arraystretch}{1.5}{\begin{tabular}{|r||r||rr||rr|}
   \hline
      \textbf{group} & \textbf{\# conj classes} & \textbf{\# length$\neq$codim} & \textbf{\# nonref atoms} & \textbf{$\dim V$} & \textbf{max ref length} \\
      \hline
      23 &  10 &  0 &  0 & 3 &  3 \\
      24 &  12 &  2 &  2 & 3 &  4 \\
      25 &  24 &  3 &  1 & 3 &  4 \\
      26 &  48 &  9 &  5 & 3 &  4 \\
      27 &  34 & 12 & 12 & 3 &  5 \\
      \hline
      28 &  25 &  0 &  0 & 4 &  4 \\
      29 &  37 & 10 &  4 & 4 &  6 \\
      30 &  34 &  0 &  0 & 4 &  4 \\
      31 &  59 & 27 &  5 & 4 &  6 \\
      32 & 102 & 27 &  6 & 4 &  6 \\
      \hline
      33 &  40 & 12 &  6 & 5 &  7 \\
      \hline
      34 & 169 & 78 & 14 & 6 & 10 \\
      35 &  25 &  0 &  0 & 6 &  6 \\
      \hline
      36 &  60 &  0 &  0 & 7 &  7 \\
      \hline
      37 & 112 &  0 &  0 & 8 &  8 \\
      \hline
   \end{tabular}   }}
   \vskip \baselineskip
	\caption{Atom count in $(\conjG,\leq_{_{\perp}})$ for exceptional reflection groups $G_{23}-G_{37}$}
	\label{tab:AtomCount}
\end{table}

\section{Conclusion}\label{CONCLUSION}
Combining the existing results for Coxeter groups and $G(m,1,n)$ with
our computations for the remaining irreducible complex
reflection groups, we complete the determination of which reflection
groups have length equal to codimension.

\begin{theorem}
   Let $G$ be an irreducible complex reflection group.  The reflection
   length and codimension functions coincide if and only if $G$ is a Coxeter
   group or $G=G(m,1,n)$.
\end{theorem}
\begin{proof}
   Carter's proof~\cite{Carter} that reflection length coincides with 
   codimension in Weyl groups works just as well for Coxeter groups,
   and 
   Shi~\cite{ShiJ} proves reflection length coincides with codimension in
   the infinite family $G(m,1,n)$ 
   (also see Shepler and Witherspoon~\cite{SheplerWitherspoonCup}
   for a more linear algebraic proof).
   For the converse, Corollary~\ref{gmpnnope} gives counterexamples for the 
   remaining groups in the family $G(m,p,n)$, while
   Proposition~\ref{ranktwonope} and Table~\ref{tab:AtomCount} 
   show reflection length and codimension
   do not coincide in the non-Coxeter exceptional complex reflection groups.
\end{proof}

\section{Applications to Cohomology}\label{COHOMOLOGY}
The codimension poset has applications to Hochschild
cohomology and deformation theory of skew group algebras
$S(V)\#G$ for a finite group $G$ acting linearly on $V$.  
Deformations of skew group algebras include 
graded Hecke algebras and symplectic reflection algebras.
Hochschild cohomology detects potential deformations.
For a $\CC$-algebra $A$ and an $A$-bimodule $M$, the
Hochschild cohomology of $A$ with coefficients in $M$ is the space
$$
\HH^{\bullet}(A,M)=\ext_{A\otimes A^{\text{op}}}^{\bullet}(A,M),
$$
where we tensor over $\CC$.
When $M=A$, we simply write $\HH^{\bullet}(A)$.  
We refer the reader to Gerstenhaber and Schack~\cite{GerstenhaberSchack} for more on
algebraic deformation theory and Hochschild cohomology.  

In the setting of skew group algebras, Hochschild cohomology may
be formulated in terms of invariant theory.
\c{S}tefan~\cite{Stefan} finds cohomology of the skew group algebra
$S(V)\#G$ as the space of $G$-invariants in
a larger cohomology ring:
$$\HH^{\bullet}(S(V)\#G)\cong\HH^{\bullet}(S(V),S(V)\#G)^G,$$
and Farinati~\cite{Farinati} and Ginzburg and Kaledin~\cite{GinzburgKaledin}
describe the larger cohomology ring:
$$
\HH^{\bullet}(S(V),S(V)\#G) \cong \bigoplus_{g \in G}
\Bigl(\:
   S(V^{g}) \otimes 
   \bigwedge^{\bullet-\codim(g)}(V^g)^{*} \otimes
   \bigwedge^{\codim(g)}\bigl((V^g)^*\bigr)^{\perp} \otimes
   \mathbb{C}g
\:\Bigr),
$$
which we identify with a subspace of $S(V) \otimes \bigwedge^{\bullet} V^* \otimes \CC G$.
Here, $V^g=\{v \in V\::\:gv=v\}$ denotes the fixed point space of $g$.

Shepler and Witherspoon~\cite{SheplerWitherspoonCup}
further show that the cohomology $\HH^{\bullet}(S(V),S(V)\#G)$ is generated
as an algebra under cup product by $\HH^{\bullet}(S(V))$ 
together with derivation forms corresponding to atoms in the codimension poset. 
More specifically,
for each $g$ in $G$, fix a choice of volume form
$\vol{g}$ in the one-dimensional space $\bigwedge^{\codim(g)}\bigl((V^g)^*\bigr)^{\perp}$.
(If $s$ is a reflection, we may take $\vol{s}$ in $V^{*}$ to be
a linear form defining the hyperplane about which $s$ reflects.)
Then ~\cite[Corollary 9.4]{SheplerWitherspoonCup}
asserts that the cohomology ring
$\HH^{\bullet}(S(V),S(V)\#G)$ is generated by
$\HH^{\bullet}(S(V))\cong S(V)\otimes\bigwedge^{\bullet}V^*\otimes 1_{G}$ and the set of volume forms tagged by codimension atoms:
$$
\{1\otimes\vol{g}\otimes g\::\:g \text{ is an atom in the codimension poset for $G$}\}.
$$

\begin{example}
Consider the group $G=G(m,p,n)$ acting on $V\cong\CC^n$ by its standard
reflection representation.  Let $v_1,\ldots,v_n$ denote the standard basis
of $V$ and $v_1^*,\ldots,v_n^*$ the dual basis of $V^*$.
As in Section~\ref{GMPN}, let $\zeta_m=e^{2\pi i/m}$.

Proposition~\ref{GmpnAtoms} describes the codimension atoms
in $G(m,p,n)$, and we can easily find the corresponding volume forms
$\vol{g}$.  
The cohomology $\HH^{\bullet}(S(V),S(V)\#G(m,p,n))$
is thus generated as a ring under cup product by $\HH^{\bullet}(S(V))$ and the elements
\begin{itemize}
  \item $1\otimes(v_i^*-\zeta_m^c v_j^{*}) \otimes s$, where $s$ is a reflection about the hyperplane $v_i^*-\zeta_m^c v_j^*=0$, and
  \item $1\otimes v_{i_1}^*\wedge \cdots \wedge v_{i_{\codim(g)}}^* \otimes g$, 
        where $g$ is $p$-connected and $v_{i_1},\ldots,v_{i_{\codim(g)}}$ form a basis of $(V^g)^{\perp}$.
\end{itemize}
(Note that we have included the elements $1\otimes(v_{i_1}^*\wedge v_{i_2}^*)\otimes g$ with $\det(g)=1$,
but these do not arise from codimension atoms and are superfluous generators.)
\end{example}

Shepler and Witherspoon~\cite[Corollary 10.6]{SheplerWitherspoonCup}
show that if $G$ is a Coxeter group or $G=G(m,1,n)$, then,
in analogy with the Hochschild-Kostant-Rosenberg Theorem,
the cohomology $\HHSSG{\mol}$ is generated
in cohomological degrees $0$ and $1$.  
We use our comparison of the reflection length
and codimension posets to show this analogue fails for 
the other irreducible complex reflection groups.

We recall from~\cite[Section 8]{SheplerWitherspoonCup} the volume algebra
$A_{\text{vol}}:=\Span_{\CC}\{1\otimes\vol{g}\otimes g\::\:g\in G\}$, isomorphic to
a (generalized) twisted group algebra with multiplication
$$
(1\otimes\vol{g}\otimes g)\smile(1\otimes\vol{h}\otimes h)=\theta(g,h)(1\otimes\vol{gh}\otimes gh)
$$
for some cocycle $\theta:G\times G\rightarrow\CC$.
The cocycle $\theta$ is generalized in that its values may include zero;
in fact, the {\sffamily\bfseries twisting constant} $\theta(g,h)$ is nonzero if
and only if $g\leq_{_\perp}gh$.  Iterating the product formula, we find
$$(1\otimes\vol{g_1}\otimes g_1)\smile\cdots\smile(1\otimes\vol{g_k}\otimes g_k)
=\lambda(1\otimes\vol{g_1\cdots g_k}\otimes g_1\cdots g_k),$$
where $\lambda=\theta(g_1,g_2)\theta(g_1g_2,g_3)\cdots\theta(g_1\cdots g_{k-1},g_k)$.
The twisting constant $\lambda$ is nonzero if and only if 
$g_1\leq_{_{\perp}} g_1g_2 \leq_{_\perp}\cdots\leq_{_{\perp}}g_1\cdots g_k$.
We make use of this fact in the proof of Lemma~\ref{supportlemma} below.

Once a choice of volume forms $\vol{g}$ has been made, then given an
element $\alpha$ in $\HHSSG{\mol}$, there
exist unique elements $\alpha_g$ in $S(V^{g})\otimes\bigwedge^{\mol}(V^{g})^*$
such that
$$
\alpha=\sum_{g\in G}\alpha_g\otimes\vol{g}\otimes g.
$$
Let the {\sffamily\bfseries support} of $\alpha$ be $\supp(\alpha)=\{g\in G\::\:\alpha_g\neq0\}$.
For a set $B\subset\HHSSG{\mol}$, let $\supp(B)=\bigcup_{\beta\in B}\supp(\beta)$.
In the next lemma, we relate the support of a subring of $\HHSSG{\mol}$
to the support of a set of generators for the subring.

\begin{lemma}\label{supportlemma}
  Let $B$ be a subring of $\HHSSG{\mol}$, and let $\mathcal{G}(B)$
  be a set of generators for $B$ as a ring under cup product.  
  If $g$ is in $\supp(B)$, then there exist group elements $g_1,\ldots,g_k$ 
  in $\supp(\mathcal{G}(B))$
  such that $g_1\leq_{_\perp}g_1g_2\leq_{_{\perp}}\cdots\leq_{_{\perp}}g_1\cdots g_k=g$.
\end{lemma}

\begin{proof}
First consider the support of
a finite cup product $\beta_1\smile\cdots\smile\beta_k$
of generators $\beta_i$ from $\mathcal{G}(B)$.  
Using the cup product 
formula~\cite[Equation (7.4)]{SheplerWitherspoonCup}\footnote{Note that the
factor $dv_g\wedge dv_h$ in Equation (7.4) may not a priori be an element of $\bigwedge^{\mol} (V^{gh})^*$.
To interpret the equation correctly, we must apply to the wedge product $dv_g\wedge dv_h$ the projection
$\bigwedge^{\mol}V^*\rightarrow\bigwedge^{\mol}(V^{gh})^*$
induced by the orthogonal projection $V^*\rightarrow (V^{gh})^*$.
After the last iteration of the cup product formula, we also apply the
projections $S(V)\rightarrow S(V)/I((V^{g})^{\perp})\cong S(V^{g})$ to the polynomial parts to obtain
a representative in $\HHSSG{\mol}$.},
we find that a typical summand of $\beta_1\smile\cdots\smile\beta_k$ has the form
$$
\omega \otimes \theta(g_1,g_2)\theta(g_1g_2,g_3)\cdots\theta(g_1\cdots g_{k-1},g_k)\vol{g}\otimes g,
$$
where each $g_i$ is in $\supp(\beta_i)$, $g=g_1\cdots g_k$, and $\omega$ is a
(possibly zero) derivation form in 
$S(V^g)\otimes \bigwedge^{\mol}(V^{g})^*$.
The scalar
$$\theta(g_1,g_2)\theta(g_1g_2,g_3)\cdots\theta(g_1\cdots g_{k-1},g_k)$$
is a twisting constant from the volume algebra and, as noted above,
is nonzero if and only if
$g_1\leq_{_{\perp}}g_1g_2\leq_{_{\perp}}\cdots\leq_{_{\perp}}g_1\cdots g_k$.
Thus $$\supp(\beta_1\smile\cdots\smile\beta_k)\subseteq\{g_1\cdots g_k:g_i\in\supp(\beta_i)\text{ and }g_1\leq_{_{\perp}}g_1g_2\leq_{_{\perp}}\cdots\leq_{_{\perp}}g_1\cdots g_k\}.$$
Now note that for arbitrary elements $\alpha_1,\ldots,\alpha_k$ in $B$, we have
$$
\supp(\alpha_1+\cdots+\alpha_k)\subseteq\supp(\alpha_1)\cup\cdots\cup\supp(\alpha_k).
$$
This proves the lemma since every element of $B$ is a sum of finite 
cup products of elements of $\mathcal{G}(B)$.
\end{proof}

\begin{corollary}\label{codimsupport}
The set of codimension atoms for $G$ is contained in the support 
of every generating set for $\HHSSG{\mol}$.
\end{corollary}

\begin{proof}
  Let $\mathcal{G}$ be a set of generators for $\HHSSG{\mol}$. 
  Applying Lemma~\ref{supportlemma}, we have that for each $g\neq1$ in $G$
  there exist nonidentity group elements $g_1,\ldots,g_k$ in $\supp(\mathcal{G})$ 
  such that $g_1\leq_{_{\perp}}g_1g_2\leq_{_{\perp}}\cdots\leq_{_{\perp}}g_1\cdots g_k=g$.
  In particular, $g_1\leq_{_{\perp}}g$. If $g$ is a codimension atom, then 
  since $g_1\neq1$ we must have $g_1=g$, and hence $g$ lies in $\supp(\mathcal{G})$.
\end{proof}

\begin{remark}
The exterior products in the description of cohomology
force a homogeneous generator supported on a group element
$g$ to have cohomological degree at least $\codim(g)$ (and no more than $\dim V=n$). 
In light of Corollary~\ref{codimsupport}, a set of
homogeneous generators for $\HHSSG{\mol}$ may well require elements
of maximum cohomological degree $n$.  For instance, in the group
$G(n,n,n)$ for $n\geq3$, the element $g=\diag(e^{2\pi i/n},\ldots,e^{2\pi i/n})$
is a codimension atom, and a homogeneous generator supported on $g$ must have
cohomological degree $\codim(g)=n$.  Thus, using Corollary~\ref{codimsupport},
we see that every set of homogeneous generators for $\HH^{\mol}(S(V),S(V)\#G(n,n,n))$
includes an element of cohomological degree $n$.
\end{remark}

\begin{corollary}
  Let $G$ be an irreducible complex reflection group.  Then the
  cohomology ring $\HH^{\bullet}(S(V),S(V)\#G)$ is generated
  in cohomological degrees $0$ and $1$
  if and only if $G$ is a Coxeter group or a monomial reflection group $G(m,1,n)$.
\end{corollary}

\begin{proof}
  By Corollary~\ref{codimsupport}, the support in $G$ of a set of generators for
  $\HHSSG{\mol}$ must contain the set of codimension atoms.  It follows
  that any set of generators contains elements of cohomological 
  degree at least as great as the codimensions of the atoms in the
  codimension poset.
  If $G$ is not a Coxeter group and not a monomial reflection group $G(m,1,n)$,
  then there are nonreflection atoms
  in the codimension poset, so a generating set for $\HHSSG{\mol}$
  will necessarily include elements of cohomological degree greater than one.
  
  Conversely, Shepler and Witherspoon show in~\cite[Corollary 10.6]{SheplerWitherspoonCup}
  that if $G$ is a Coxeter group or a monomial reflection group
  $G(m,1,n)$, then $\HHSSG{\mol}$ can in fact be generated in degrees
  $0$ and $1$.
\end{proof}


\bibliographystyle{abbrv}
\bibliography{msreferences}

\end{document}